\documentclass{article}

\usepackage[final]{nips_2017}

\usepackage[utf8]{inputenc} 
\usepackage[T1]{fontenc}    
\usepackage{hyperref}       
\usepackage{url}            
\usepackage{booktabs}       
\usepackage{amsfonts}       
\usepackage{nicefrac}       
\usepackage{microtype}      

\hypersetup{citecolor=blue}
\usepackage{enumitem}
\usepackage{amsmath}
\usepackage{amsthm}
\usepackage{thmtools,thm-restate}
\usepackage{amsfonts}
\usepackage{algorithm}
\usepackage{algpseudocode}
\usepackage{natbib}
\usepackage{hyperref}
\hypersetup{
  colorlinks = true
}
\usepackage{mathtools}
\usepackage{cleveref}
\crefname{equation}{Eq.}{Eqs.}

\declaretheorem[name=Lemma]{lem}
\declaretheorem[name=Theorem]{thm}
\declaretheorem[name=Corollary]{cor}

\DeclareMathOperator{\diam}{diam}

\DeclareMathOperator{\grad}{\nabla}

\DeclareMathOperator*{\argmin}{arg\,min}

\newcommand{\E}{{\mathbb E}}

\DeclareMathOperator{\sign}{sgn}
\DeclarePairedDelimiter{\norm}{\lVert}{\rVert}

\DeclarePairedDelimiterX{\divx}[2]{(}{)}{%
  #1\;\delimsize\|\;#2%
}

\newcommand{\var}{w}
\newcommand{\atoms}{S}
\newcommand{\cohull}{W}

\newcommand{\showthmname}[1]{}
\newcommand\numberthis{\addtocounter{equation}{1}\tag{\theequation}}

\title{Approximate and Stochastic Greedy Optimization}

\author{
  Nan Ye \\
	QUT \& ACEMS\\
  \texttt{n.ye@qut.edu.au} \\
  \And
  Peter Bartlett\\
	UC Berkeley \& QUT \& ACEMS\\
  \texttt{bartlett@cs.berkeley.edu} \\
}

\usepackage{thmtools}

\begin{document}

\maketitle

\begin{abstract}
We consider two greedy algorithms for minimizing a convex function in a
bounded convex set:
an algorithm by \citet{jones1992simple} and the Frank-Wolfe (FW) algorithm.
We first consider approximate versions of these algorithms.
For smooth convex functions, we give sufficient conditions for convergence,
a unified analysis for the well-known convergence rate of $O(1/k)$ together
with a result showing that this rate is the best obtainable from the
proof technique, and an equivalence result for the two algorithms. 
We also consider approximate stochastic greedy algorithms for minimizing expectations.
We show that replacing the full gradient by a single stochastic gradient can
fail even on smooth convex functions.
We give a convergent approximate stochastic Jones algorithm and a convergent
approximate stochastic FW algorithm for smooth convex functions.
In addition, we give a convergent approximate stochastic FW algorithm
for nonsmooth convex functions.
Convergence rates for these algorithms are given and proved.
\end{abstract}

\section{Introduction}

Consider the following problem of minimizing a convex function over a convex set,
\begin{align}
  \min_{\var \in \cohull} f(\var),
\end{align}
where $\cohull$ is the convex hull of a set of atoms $\atoms$ in a linear
vector space.
Such problem occurs frequently in machine learning and engineering
\citep{boyd2004convex}.
We consider greedy algorithms which starts with some 
$\var_1 \in \cohull$, and then iteratively find 
$\var_{k+1} = (1 - \eta_{k}) \var_{k} + \eta_{k} d_{k}$, where 
$\eta_{k}$ and/or $d_{k} \in \atoms$ are greedily chosen according to certain
criterion.
An attractive feature of such algorithm is that the iterates are sparse,
because each iteration adds at most one new atom in $\atoms$.

Two greedy algorithms are well-known:
an algorithm originally studied by \citet{jones1992simple},
and the Frank-Wolfe (FW) algorithm \citep{frank1956algorithm}.
Jones' algorithm chooses
\begin{align*}
	(\eta_{k}, d_{k}) 
	= \argmin_{\eta \in [0,1], d \in \atoms} f(\eta \var_{k} + (1 - \eta) d).
\end{align*}
This has been studied in various contexts,
such as function approximation in the Hilbert space
\citep{jones1992simple,barron1993universal,lee1996efficient}, 
$\ell_p$ regression \citep{donahue1997rates}, 
density estimation \citep{li1999mixture}, and is closely related to
boosting \citep{zhang2003sequential}.
The FW algorithm chooses
\begin{align*} 
	d_{k} = \argmin_{d \in \atoms} \grad f(\var_{k})^{\top} d,
\end{align*}
and chooses $\eta_k$ by line search or a priori.
The FW algorithm has recently attracted significant interest due to
its projection-free property and the ability to handle structural constraints
\citep{jaggi2013revisiting}.
In contrast to solving quadratic programs for projection in projected gradient
descent and for the proximal map in the proximal algorithms, the FW
algorithm solves a linear program at each step, which is often computationally
more tractable \citep{jaggi2010simple,lacoste2013affine}.
Approximate versions of Jones' algorithm and the FW algorithm have also been 
studied, for example, see \citep{zhang2003sequential,jaggi2013revisiting}.

In this paper, we first consider approximate versions of Jones' algorithm and
the FW algorithm, with a more general approximate version for Jones'
algorithm.
We focus on smooth convex functions in our analysis, and give a sufficient
convergence condition for both algorithms\showthmname{thm:converge}.
Building on previous results on the $O(1/k)$ convergence rates for both
algorithms, we present a unified analysis for the $O(1/k)$ convergence
rate\showthmname{thm:rate}, and also show that this is the optimal that can be
obtained with the proof technique\showthmname{thm:optimal}.
We also show that the approximate Jones' algorithm and the approximate FW
algorithm are equivalent\showthmname{thm:equiv}.

We then consider stochastic versions of these approximate greedy algorithms
for the stochastic approximation problem, where $f$ is an expectation $\E
f_{z}(\var)$ over some random variable $z$.
We show that some stochastic versions fail even on smooth convex functions.
We give an approximate stochastic Jones algorithm that has error $\epsilon$
using $O(\epsilon^{-4})$ random $f_{z}(\var)$ for smooth convex functions.
We also give an approximate stochastic FW algorithm that has an error
$\epsilon$ using $O(\epsilon^{-4})$ stochastic gradients and
$O(\epsilon^{-2})$ linear optimizations.
In addition, we give an approximate stochastic Frank-Wolfe algorithm that has
error $\epsilon$ using $O(\epsilon^{-4})$ stochastic gradients for nonsmooth
convex functions.
The algorithms also apply to the finite-sum setting where 
$f(\var) = \frac{1}{n} \sum_{i=1}^n f_i(\var)$.
The finite-sum form occurs when performing empirical risk minimization in
machine learning, or when performing M-estimation in statistics.
In both cases, each $f_i$ measures how well a model fits an example.

Stochastic algorithms originated in the 1950s \citep{robbins1951stochastic},
and have attracted much interest in recent years, mainly due to its ability
to scale up to large datasets.
We note that stochastic FW algorithms have recently been considered for
smooth functions by \citet{reddi2016stochastic-b} and
\citet{hazan2016variance}.
\citet{reddi2016stochastic-b} considered the non-convex setting, and shows
that one can achieve an error of $\epsilon$ with $O(\epsilon^{-4})$ stochastic
gradients and $O(\epsilon^{-2})$ linear optimizations.
When $f$ is a finite sum, the number of stochastic gradients needed can be
reduced to $O(n + n^{1/3} \epsilon^{-2})$. 
\citet{hazan2016variance} considered the convex setting, and showed that one
can achieve an error of $\epsilon$ with $O(\ln \epsilon^{-1})$ full gradients,
$O(\epsilon^{-2})$ stochastic gradients, and $O(\epsilon^{-1})$ linear
optimizations.
The number of stochastic gradients can be reduced to $O(\ln \epsilon^{-1})$
if $f$ is strongly-convex.
Both works use recent variance reduction techniques in convex optimization,
such as the works of
\citet{johnson2013accelerating,mahdavi2013mixed,defazio2014saga}.
\citet{hazan2016variance} additionally uses 
\citet{nesterov1983method}'s acceleration technique.
They use exact greedy steps, instead of approximate greedy steps as in this
paper.

For the non-stochastic case, faster rates for FW are known with additional
assumptions \citep{lacoste2015global,garber2015faster,garber2016linearly}.
We refer the readers to the works of \citet{hazan2016variance} and
\citet{reddi2016stochastic-b} for further related works.

\section{Approximate Greedy Optimization}
We consider the approximate Jones' algorithm in \Cref{alg:greedy}.
At each iteration, the algorithm solves the optimization problem
$\min_{d \in \atoms} f\left((1 - \eta_{k}) \var_{k} + \eta_{k} d\right)$
with an error of $\epsilon_{k} \eta_{k}$.
We call this an $\epsilon_{k}$-approximate Jones' algorithm, and we say the algorithm is a
$c$-Jones algorithm if there is a constant $c \ge 0$ such that
$\epsilon_{k} \le c \eta_{k}$ for all $k$.

\begin{algorithm}[h!]
\caption{Approximate Jones' Algorithm}
\label{alg:greedy}
\begin{algorithmic}
\State Choose $\var_0 \in \cohull$.
\For{$k = 0, 1, 2, \ldots$}
	\State Choose $(\eta_{k}, d_{k}) \in [0, 1] \times \atoms$ at iteration $k$ such that
		\begin{align}
			f\left((1 - \eta_{k}) \var_{k} + \eta_{k} d_{k} \right) \le \min_{d \in \atoms} 
				f\left((1 - \eta_{k}) \var_{k} + \eta_{k} d\right) + \epsilon_{k} \eta_k.
		\end{align}
	\State $\var_{k+1} = (1 - \eta_{k}) \var_{k} + \eta_{k} d_{k}$.
\EndFor
\end{algorithmic}
\end{algorithm}

We leave the choice of $\eta_{k}$ unspecified,
and thus this includes algorithms which fix $\eta_{k}$ a priori,
or choose $\eta_{k}$ and $d_{k}$ jointly at each iteration.
Similarly, $\epsilon_{k}$ may be chosen a priori or chosen adaptively.

An algorithm is called an $\epsilon_{k}$-approximate FW
algorithm, if given $\var_{k} \in \cohull$, the algorithm yields 
$(\eta_{k}, d_{k}) \in [0, 1] \times \atoms$ such that
\begin{align}
  \grad f(\var_{k})^\top d_{k} \le \min_{d \in \atoms} \grad f(\var_{k})^\top d + \epsilon_{k},
\end{align}
and we say the algorithm is a $c$-FW algorithm for some $c \ge 0$ if
$\epsilon_{k} \le c \eta_{k}$.

\subsection{Assumptions}
In this section, we assume $f$ is convex with bounded curvature, that is,
\begin{align}
	f(\var') \ge f(\var) + \grad f(\var)^{\top} (\var' - \var), 
		\quad\text{for all $\var', \var \in \cohull$,}
		&&& \text{(convexity)}\\
	\sup_{\var \in \cohull, d \in \atoms, \eta \in (0,1)} 
    \frac{2}{\eta^2} D_f\left((1-\eta) \var  + \eta d, \var\right) 
		< \infty,
		&&& \text{(bounded curvature)}
\end{align}
where $D_f(\var, y) = f(\var) - f(y) - \grad f(y) (\var - y)$ is the Bregman
divergence of $f$, and the LHS of the second equation is called the
\emph{curvature} of $f$ in $\cohull$.
This definition of curvature is the same as that in
\citep{jaggi2013revisiting}, except that \citet{jaggi2013revisiting} takes
supremum over $d \in \cohull$.
If $f$ is $L$-smooth, that is,
$\norm{\grad f(\var') - \grad f(\var)}_2 \le L \norm{\var' - \var}_2$ for all
$\var', \var \in \cohull$,
then the curvature of $f$ is not more than $L \diam(S)^2$.
Thus a smooth function has bounded curvature.
The curvature of $f$ is also not more than
  $\sup_{\var \in \cohull, d \in \atoms, \eta \in (0,1)} 
		\frac{\partial f((1-\eta) \var + \eta d)}{\partial \eta^2}$, 
assuming the second-order derivative exists.

The following are two basic bounds needed in our analysis.
\begin{lem}
(a) (Duality bound) If $f$ is convex on $\cohull$, 
$\var^* = \argmin_{\var \in \cohull} f(\var)$, then
for any $\var \in \cohull$, 
\begin{align}
  f(\var) - f(\var^*) 
		\le \max_{d \in \atoms} f(\var)^\top (\var - d). \label{eq:duality}
\end{align}
(b) (Curvature inequality) If $f$ has curvature at most $M$, then for any $\var \in \cohull$, 
$d \in \atoms$, $\eta \in [0, 1]$,
\begin{align}
  f\left((1 - \eta) \var + \eta d\right) 
    &\le f(\var) + \eta \grad f(\var)^\top (d - \var) + \frac{M}{2} \eta^2.
		\label{eq:curvature}
\end{align}
\end{lem}
\begin{proof}
(a) Using the definitions, we have
	\begin{align*}
		f(\var) - f(\var^*) 
     \le \grad f(\var)^\top (\var - \var^*) 
     \le \max_{d \in \cohull} f(\var)^\top (\var - d)
    \le \max_{d \in \atoms} f(\var)^\top (\var - d).
	\end{align*}

(b) From the definition of Bregman divergence, we have 
\begin{align*}
  f\left((1 - \eta) \var + \eta d\right) 
  = f(\var) + \eta \grad f(\var)^\top (d - \var) 
  + D_f\left((1 - \eta) \var + \eta d \right).
\end{align*}
Apply the definition of curvature, then the desired inequality follows.
\end{proof}

In general, we cannot improve the quadratic term to a higher-order one in the
curvature inequality.
For example, if $f$ is $m$-strongly convex, then we can show that
$\sup_{\var \in \cohull, d \in \atoms, \eta \in (0,1)} 
\frac{1}{\eta^3} D_f\left((1-\eta) \var  + \eta d, \var\right)$ is infinity.

\subsection{A Sufficient Condition for Convergence} 

The core to our convergence analysis for Jones' algorithm and the FW
algorithm is the following recurrence equation for the
error $e_{k} = f(\var_{k}) - f(\var^*)$.

\begin{restatable}{lem}{lemapproxrecurrence} \label{lem:approx-recurrence}
Let $f$ be convex with curvature at most $M$.
Then for both $\epsilon_{k}$-approximate Jones' algorithm and
$\epsilon_{k}$-approximate FW algorithm the error 
$e_{k} = f(\var_{k}) - f(\var^*)$ satisfies
\begin{align}
	e_{k+1} \le (1 - \eta_k) e_{k} + \epsilon'_{k},
\end{align}
where $\epsilon'_{k} = \eta_k \epsilon_k + \frac{M}{2} \eta_k^2$.
\end{restatable} 
We omit the proof of this lemma and a few other proofs in the main text, but
put them in the supplementary material, due to space limit.

The above lemma leads to a general convergence result for Jones' algorithm and
the FW algorithm.
\begin{restatable}{thm}{converge} \label{thm:converge}
Let $f$ be convex with curvature at most $M$.
For an $\epsilon_{k}$-approximate Jones' algorithm or an
$\epsilon_{k}$-approximate FW algorithm, if
$\eta_{k}$'s and $\epsilon_{k}$'s are chosen such that $\sum_{k} \eta_{k}$
diverges, $\eta_{k} \to 0$ and $\epsilon_{k} \to 0$ as $k \to \infty$, then 
$f(\var_{k}) \to f(\var^*)$ as $k \to \infty$.
\end{restatable}
\begin{proof}
From \Cref{lem:approx-recurrence}, it suffices to show that under the given
conditions on $\eta_{k}$ and $\epsilon_{k}$, the solution to the recurrence
equation $e_{k+1} \le (1 - \eta_{k}) e_{k} + \epsilon'_{k}$ satisfies 
$e_{k} \to 0$.

For any $\delta$ such that $0 < \delta < 1$, there exists $K$ such that
for all $k > K$, we have 
	$\eta_{k} < \delta$, 
	$\epsilon'_{k} / \eta_{k} < \delta / 2$, 
because $\eta_{k} \to 0$ and $\epsilon_{k} / \eta_{k} \to 0$.
For any $k > K$, if $e_{k} > \delta$, then we have
\begin{align*}
  e_{k+1} \le e_{k} + \eta_{k} (\epsilon'_{k} / \eta_{k} - e_{k})
    \le e_{k} - \eta_{k} \delta / 2. 
\end{align*}
Since $\sum_{k} \eta_{k}$ diverges, thus if $e_{k} > \delta$, then there exists
$N > K$ such that $e_{N} \le \delta$.
We show by induction that all $k \ge N$, we have $e_{k} \le \delta$.
This is true for $k = N$.
For the inductive case, assume $e_{k} \le \delta$.
If $e_{k} \ge \delta / 2$, then $\epsilon'_{k}/\eta_{k} - e_{k} \le 0$, and
thus
$e_{k+1} \le e_{k} + \eta_{k} ( \epsilon'_{k}/\eta_{k} - e_{k}) \le \delta$.
If $e_{k} \le \delta / 2$, then
$e_{k+1} \le \frac{\delta}{2} + \delta ( \frac{\delta}{2} - 0) \le \delta$.
We have thus proved that for any $\delta > 0$, there exists $N$ such that for
all $k \ge N$, $e_{k} \le \delta$.
Thus $e_{k} \to 0$ as $k \to \infty$.
\end{proof}

\subsection{Convergence Rate}
We now show that with proper choices of $\eta_{k}$'s and $\epsilon_{k}$'s, we can
obtain a convergence rate of $O(1/k)$ for Jones' algorithm and the FW algorithm.
\begin{restatable}{thm}{thmrate} \label{thm:rate}
Let $f$ be convex with curvature at most $M$,
$\eta_{k} = \frac{2}{k+2}$ for $k \ge 0$.
Then for the iterates $(\var_{k})$ obtained
using a $c$-Jones algorithm or a $c$-FW algorithm, when $k \ge 1$,
\begin{align}
  f(\var_{k}) - f(\var^*) \le \frac{2M+4c}{k+2}.
\end{align}
\end{restatable}

The constant in the rate can be improved in some cases.
For example, if the minimizer is an algebraic interior point, then we can
get a smaller constant using an argument similar to that in
\citep{zhang2003sequential}.

A careful look at the analysis shows that if our update rule is guaranteed to
generate a new iterate that is not more than that generated by a $c$-FW
algorithm or a $c$-Jones algorithm with step size $\eta_k = 2/(k+2)$, then we
can get an $O(1/k)$ convergence rate.
This also implies that we can mix $c$-FW steps and $c$-Jones steps to get an
$O(1/k)$ convergence rate.
In addition, we can obtain the following result from \citet{zhang2003sequential} as
a special case.

\begin{cor}
Let $f$ be convex with curvature at most $M$.
If $\var_{k+1} \in \cohull$ is chosen such that 
\begin{align*} 
	f(\var_{k+1})
	\le \min_{\eta \in [0,1], d \in \atoms} 
	f\left((1 - \eta) \var_{k} + \eta d\right) + \frac{4c}{(k+2)^2},
\end{align*}
where $c > 0$ is some constant, then for $k \ge 1$, we have
$f(\var_{k}) - f(\var^*) \le \frac{2M + 4c}{k+2}$.
\end{cor}

The key idea in the above analysis is to show that 
$e_{k+1} \le (1 - \eta_{k}) e_{k} + C \eta_{k}^2$,
and then use induction to show that $e_{k} \in O(\frac{1}{k})$ when
$\eta_{k} = \frac{2}{k+2}$.
Can we tune $\eta_{k}$ to obtain a bound $O(\frac{1}{k^p})$ for some $p > 1$?
It turns out that $p = 1$ is the best obtainable.
\begin{thm} \label{thm:optimal}
Consider a sequence $(e_{k})$ satisfying 
\begin{align}
  e_{k+1} = (1 - \eta_{k}) e_{k} + C \eta_{k}^2,
\end{align}
with $e_0 \le 2C$, then for any choice of $\eta_{k}$, we have 
$e_{k} \ge \frac{a}{k+2}$ for $a = \min\{e_0, C\}$.
\end{thm}
\begin{proof}
Clearly $e_0 \ge \frac{a}{2}$ holds.
Now we show by induction that if $e_{k} \ge \frac{a}{k+2}$, then 
$e_{k+1} \ge \frac{a}{k+3}$.
Note that $(1 - \eta_{k}) e_{k} + C \eta_{k}^2$ is minimized when 
$\eta_{k} = \frac{e_{k}}{4C}$, with minimum value $e_{k} (1 - \frac{e_{k}}{4C})$,
which is an increasing function of $e_{k}$ when 
$e_{k} \in [\frac{a}{k+2}, 2C]$.
This implies that when $\eta_{k}$'s are chosen to minimize $e_{k}$'s, then $e_{k}$'s
form a decreasing sequence.
Since $e_0 \le 2C$, this also implies the minimum $e_{k} \le 2C$.
Hence we have
\begin{align*}
  e_{k+1} 
    \ge e_{k} \left(1 - \frac{e_{k}}{4C}\right)
    \ge \frac{a}{k+2} \left(1 - \frac{a/4C}{k+2}\right)
    \ge \frac{a(k+2 - a/4C)}{(k+2)^2}
    \ge \frac{a}{k+3},
\end{align*}
where the last inequaliy holds because
\begin{align*} 
  (k+2 - a/4C)(k+3) 
  \ge (k + 2 - 1/4)(k+3)
  \ge k^2 + \frac{19}{4} k + \frac{21}{4}
  \ge (k+2)^2.
\end{align*}
\end{proof}

\subsection{An Equivalence Result}
We have already seen that a few results hold for both the approximate Jones'
algorithm and the approximate FW algorithm.
The following theorem shows that we can view these two algorithms as
equivalent algorithms.

\begin{restatable}{thm}{thmequiv} \label{thm:equiv}
Assume $f$ is convex with curvature at most $M$.
\begin{itemize}[leftmargin=1.5em]
\item[(a)] An $\epsilon_{k}$-approximate Jones' algorithm with step sizes $(\eta_{k})$ is 
$(\epsilon_{k} + \frac{M}{2} \eta_{k})$-FW.
In particular, a $c$-Jones algorithm is $\frac{M+2c}{2}$-FW with the same step sizes.
\item[(b)] An $\epsilon_{k}$-approximate FW algorithm with step sizes $(\eta_{k})$ is 
an $(\epsilon_{k} + \frac{M}{2} \eta_{k})$-approximate Jones' algorithm.
In particular, a $c$-FW algorithm is $\frac{M+2c}{2}$-Jones with the same step
sizes.
\end{itemize}
\end{restatable}

An immediate consequence of this result is that if any $c$-Jones algorithm
converges at $O(1/k)$ rate, then any $c$-FW algorithm converges at $O(1/k)$
rate too.

\section{Approximate Stochastic Greedy Optimization}

We consider approximate stochastic versions of Jones' algorithm and the FW
algorithm for optimizing a function $f(\var) = \E f_z(\var)$, where the
expectation is over a random variable.
Without loss of generality, we work with the finite-sum case where 
$f(\var) = \frac{1}{n} \sum_{i=1}^{n} f_i(\var)$ to ease presentation.

\vspace{-0.5em}
\subsection{Stochastic Jones' Algorithm}
A natural stochastic version of Jones' algorithm is
obtained by replacing the function $f$ with a sampled approximation
$\tilde{f}_{k}$ at iteration $k$.

\begin{algorithm}[h!]
	\caption{Approximate Stochastic Jones (ASJ)}  
	\label{alg:sg}
	\begin{algorithmic}
	\State Choose $\var_1 \in \cohull$.
	\For{$k = 1, 2, \ldots$}
		\State Sample a set $I_{k}$ of $b_k$ numbers independently and uniformly from $[n]$, and let
			\begin{align}
				\tilde{f}_{k}(\var) = \frac{1}{b_k} \sum_{i \in I_{k}} f_i(\var).
			\end{align}
		\State Choose $(\eta_{k}, d_{k}) \in [0,1] \times \atoms$ such that 
			\begin{align}
				\tilde{f}_{k}\left((1 - \eta_{k}) \var_{k} + \eta_{k} d_{k} \right) \le \min_{d \in \atoms} 
					\tilde{f}_{k}\left((1 - \eta_{k}) \var_{k} + \eta_{k} d\right) + \eta_{k} \epsilon_{k}.
			\end{align}
		\State $\var_{k+1} = (1 - \eta_{k}) \var_{k} + \eta_{k} d_{k}$.
	\EndFor
	\end{algorithmic}
\end{algorithm}

We show that ASJ is over-greedy when $b_k=1$ and the minimization
problem at each iteration is solved exactly.
The iterates can jump randomly from one vertex to another, leading to divergence.
This differs from the nonstochastic case where exact minimization leads to
smaller errors.
\begin{restatable}{prop}{propasjeg} \label{prop:asj-eg}
Let $b_{k} = 1$, $\epsilon_{k} = 0$ and $\eta_{k}$ jointly optimized with
$d_{k}$ in ASJ, then there exists a function 
$f(\var) = \frac{1}{n} \sum_{i=1}^{n} f_i(\var)$ with each $f_i$ being convex
and smooth, such that $\E f(\var_k) - f(\var^*)$ does not converge to 0 as 
$k \to \infty$.
\end{restatable}
On the other hand, we can get a convergent algorithm using increasingly
larger batch size.
In essence, the theorem below shows that when we choose a batch size of $k$ a
iteration $k$ with a step size $\sqrt{k}$, we can get an error of
$O(1/\sqrt{t})$ at any iteration $t$.
Taking $b_{k}$ as a measure of the computational complexity of the $k$-th
problem, then to get an error of $\epsilon$, the complexity of the algorithm
is $O(\epsilon^{-4})$.

\begin{restatable}{thm}{thmasj} \label{thm:asj}
Assume that the diameter of $\cohull$ is $D$, each
$f_i(\var)$ is convex with curvature at most $M$, and 
$\norm{\grad f_i(\var)}_2 \le L$ for all $i$ and $\var \in \cohull$.
Let $\bar{\var}_{k} = \sum_{i=1}^{k} \eta_{i} \var_{i} / \sum_{i=1}^{k} \eta_{i}$.
In ASJ, when $b_{k} = t$, $\eta_{k} = t^{-1/2}$ and 
$\epsilon_{k} = c \eta_{k}$ for all $k$, we have 
\begin{align}
	\E f(\bar{\var}_{t}) - f(\var^*)
	\le \frac{f(\var_1) - f(\var^*) + DL + M + c}{\sqrt{t}},
\end{align}
When $b_k = k$, and $\eta_{k} = k^{-1/2}$, we have 
\begin{align}
	\E f(\bar{\var}_{k}) - f(\var^*)
	\le \frac{f(\var_1) - f(\var^*) + (DL + M + c) (\ln t + 1)}{\sqrt{t}}.
\end{align}
\end{restatable}

\subsection{Approximate Stochastic Versions of Frank-Wolfe}

For FW, we can also sample a mini-batch estimation of the function $f(\var)$
and use the gradient of the estimation to replace the gradient of $f$, as
shown in \Cref{alg:asfw}.

\begin{algorithm}
	\caption{Approximate Stochastic Frank Wolfe (ASFW)}
	\label{alg:asfw}
	\begin{algorithmic}
	\State Choose $\var_0 \in \cohull$.
	\For{$k = 0, 1, 2, \ldots$}
		\State Sample a set $I_{k}$ of $b_k$ numbers independently and uniformly from $[n]$, and let
			\begin{align}
				\tilde{f}_{k}(\var) = \frac{1}{b_k} \sum_{i \in I_{k}} f_i(\var).
			\end{align}
		\State Choose $d_{k} \in \atoms$ such that
			\begin{align}
				\grad \tilde{f}_{k}(\var_{k})^{\top} d_{k} 
					\le \min_{d \in \atoms} \grad \tilde{f}_{k}(\var)^{\top} d + \epsilon_{k}.
			\end{align}
		\State $\var_{k+1} = (1 - \eta_{k}) \var_{k} + \eta_{k} d_{k}$.
	\EndFor
	\end{algorithmic}
\end{algorithm}

We can show that if there exists a constant $c > 0$, for all $k \ge 0$, we have
\begin{align}
  \langle \E(d_{k}), \grad f(\var_{k}) \rangle 
    \le \min_{d \in \atoms} \langle d, \grad f(\var_{k}) \rangle + c \eta_{k},
	\label{eq:expect}
\end{align}
then $\E(f(\var_k)) - f(\var^*)$ is of the order $O(1/k)$ for $k \ge 1$.
The above recursive property is a sufficient but not necessary condition for
ASFW to have $O(1/k)$ convergence rate.
Indeed, there are cases where the above recursive property does not hold, but
ASFW converges.

\begin{restatable}{prop}{propasfwega} \label{prop:asfw-eg-a}
Let $b_{k} = 1$, $\epsilon_{k} = 0$, $\eta_{k} = \frac{2}{k+2}$ in ASFW.
There exists a function $f(\var) = \frac{1}{n} \sum_{i=1}^{n} f_i(\var)$ with
each $f_i$ being convex and smooth, such that 
$\E f(\var_{k}) \to f(\var^*)$ as $k \to \infty$ but 
\Cref{eq:expect} is not satisfied.
\end{restatable}

\begin{restatable}{prop}{propasfwegb} \label{prop:asfw-eg-b}
Let $b_{k} = 1$, $\epsilon_{k} = 0$, and $\eta_{k}$ be arbitrarily chosen in
ASFW.
There exists a convex and smooth $f$ such that 
$\lim_{k \to \infty} \E f(\var_{k})$ exists, but the limit is
larger than $f(\var^*)$.
\end{restatable}

\citet{reddi2016stochastic-b} considered the exact version of ASFW, that is,
the case with $\epsilon_{k} = 0$.
They showed that for smooth nonconvex $f$, with suitale
choice of $b_{k}$ and $\eta_{k}$, one can achieve an error of $\epsilon$ with
$O(\epsilon^{-4})$ stochastic gradients and $O(\epsilon^{-2})$ linear
optimizations.
We remark here that we can generalize their results to the approximate case:
we choose $b_{k}$, $\epsilon_{k}$, $\eta_{k}$ as in \Cref{thm:asj}, then we
get the same kind of bound as for ASJ, with difference only in the constants. 
This result applies to both the smooth convex case and the smooth
nonconvex case, with the cost in the nonconvex case having the form of the
duality bound.

We consider the nonsmooth convex case, and give a stochastic version that has
error $\epsilon$ using $O(\epsilon^{-4})$ stochastic gradients and
$O(\epsilon^{-4})$ linear optimizations.
The algorithm aggregates past stochastic gradients to construct a proxy 
$\bar{g}_{k} + \grad \Phi(\var_{k})$ for the full gradient.
The component $\bar{g}_{k}$ is a weighted sum of the stochastic gradients from
past iterations.
The term $\grad \Phi(\var_{k})$ has a regularizing effect of encouraging
alignment of $d$ with $\var_{k} - \var_1$ when $\Phi$ is strongly convex with
$\Phi(\var_1) = 0$. 
This is because 
$\langle \grad \Phi(\var) - \grad \Phi(\var'), \var - \var' \rangle
\ge \rho \norm{\var - \var'}_2^2$.
Without loss of generality, assume $\Phi(\var)$ is $\rho$-strongly convex and
1-smooth.
One possible choice of $\Phi$ is 
$\Phi(\var) = \frac{1}{2} \norm{\var - \var_1}_2^2$.

\begin{algorithm}[h!]
	\caption{Approximate Regularized Stochastic Frank Wolfe}
	\label{alg:arsfw}
	\begin{algorithmic}
		\State $\var_1 \gets \argmin_{\var \in \cohull} \Phi(\var)$.
		\State $\bar{g}_1 = 0$.
		\For{$k = 1, 2, \ldots$}
			\State Choose $d_{k} \in \atoms$ such that
				$\left(\bar{g}_{k} + \grad \Phi(\var_{k})\right)^{\top} d_{k}
				\le \min_{d \in \atoms} 
				\left(\bar{g}_{k} + \grad \Phi(\var_{k})\right)^{\top} d + \epsilon_{k}$.
			\State $\var_{k+1} = (1 - \eta_{k}) \var_{k} + \eta_{k} d_{k}$.
			\State $\bar{g}_{k+1} = \bar{g}_{k} + \sigma_{k} g_{k}$, where
			$g_{k} = \grad f_{i_{k}}(\var_{k})$ for random $i_{k}$ in $[n]$.
		\EndFor
	\end{algorithmic}
\end{algorithm}

A similar algorithm has been used in online learning by 
\citet{hazan2012projection,hazan2016introduction}.
They used fixed instead of variable $\eta_{k}$, and they perform exact instead
of approximation minimization at each step.

\begin{thm} \label{thm:arsfw}
Let $\Phi(\var)$ be a $\rho$-stronly convex 1-smooth function,
$R^2 = \max_{\var \in \cohull} \Phi(\var) - \Phi(\var_1)$,
$\eta_{k} = \frac{1}{k^{p}}$, 
$\epsilon_{k} = \frac{(\lambda-1) R^2}{\rho} \eta_{k}$, and
$\sigma_{k} \le c \eta_{k}^{3/2}$, where $\lambda \ge 1$, $c > 0$ and 
$p \in [0,1]$ are constants. 
Assume $\norm{\grad f_i(\var)}_2 \le L$ for all $i$ and $\var$.
Let $\bar{\var}_t = \sum_{k=1}^{t} \sigma_{k} \var_{k} / \sum_{k=1}^{t} \sigma_{k}$,
	$K = \left(\sqrt{\frac{1}{2\rho}} \frac{cL}{1-p}
		+ \sqrt{\frac{\lambda R^2 + c^2 L^2}{\rho(1-p)} 
		+ \frac{c^2 L^2}{2 \rho (1-p)^2}}\right)^2$,
then we have
\begin{align}
  \E f(\bar{\var}_t) - f(\var^*) 
	\le 
		\left(3L \sqrt{\frac{2K}{\rho}} + \frac{2c L^2}{\rho}\right)
		\frac{\sum_{k=1}^{t} \sigma_{k} \sqrt{\eta_{k}}}{\sum_{k=1}^{t} \sigma_{k}}
		+ \frac{R^2}{\sum_{k=1}^{t} \sigma_{k}}.
		\label{eq:rate1}
\end{align}
In particular, when $p = \frac{1}{2}$, for any $t \ge 1$,
\begin{align}
  \E f(\bar{\var}_t) - f(\var^*) 
	\le 
		\left(3L \sqrt{\frac{2K}{\rho}} + \frac{2c L^2}{\rho}\right)
		\frac{\ln t + 1}{t^{1/4}}
		+ \frac{R^2}{c t^{1/4}}.
		\label{eq:rate2}
\end{align}
In addition, if $\sigma_{k} = \sigma = \frac{c}{t^{3/4}}$, then 
\begin{align}
  \E f(\bar{\var}_t) - f(\var^*) 
	\le 
		\left(4L \sqrt{\frac{2K}{\rho}} + \frac{8 c L^2}{3 \rho} + \frac{R^2}{c}\right)
		\frac{1}{t^{1/4}}.
		\label{eq:rate3}
\end{align}
\end{thm}

We state two lemmas and then prove this theorem.

\begin{restatable}{lem}{lemlazy} \label{lem:lazy}
Let $\Phi(\var)$ be a $\rho$-strongly convex function,
$F_{k}(\var) = \sum_{i=1}^{s-1} \sigma_i g_i^{\top} \var + \Phi(\var)$, 
$\var^*_{k} = \argmin_{\var \in X} F_{k}(\var)$, 
$R^2 = \max_{\var \in \cohull} \Phi(\var) - \Phi(\var^*_1)$.
Then for any $\var \in X$,
\begin{align*}
	\sum_{k=1}^{t} \sigma_{k} g_{k}^{\top} (\var^*_{k} - \var) 
	\le \sum_{k=1}^{t} \frac{2 \sigma_{k}^2 \norm{g_{k}}_2^2}{\rho} + R^2.
\end{align*}
\end{restatable}

\begin{restatable}{lem}{lemsfwrecurrence} \label{lem:arsfw-recurrence}
Let $\eta_{k} = \frac{1}{k^{p}}$, and
$\sigma_{k} \le c \eta_{k}^{3/2}$, where $c$ is a positive constant, and $p$ a
positive constant in $(0,1)$.
Let $R$, $\rho$ and $L$ be positive constants, and $K$ as defined in
\Cref{thm:arsfw}.
If $e_1 \le K \eta_1$, and
\begin{align*}
  e_{k+1} 
    &\le (1 - \eta_{k}) e_{k}
      + \frac{\lambda R^2}{\rho} \eta_{k}^2
			+ \sqrt{\frac{2}{\rho}} \sigma_{k} L \sqrt{e_{k+1}}, \numberthis\label{eq:e}
\end{align*}
then $e_{k} \le K \eta_{k}$ for any $k \ge 1$.
\end{restatable}

\begin{proof}[Proof of \Cref{thm:arsfw}]
Let $F_{k}(\var) = \sum_{i=1}^{s-1} \sigma_i g_i^{\top} \var + \Phi(\var)$,
then $F_{k}$ is $\rho$-strongly convex and 1-smooth.
Let $\var^*_{k} = \argmin_{\var \in \cohull} F_{k}(\var)$,
$h_{k}(\var) = f_{i_{k}}(\var)$, and
$\tilde{h}_{k}(\var) = h_{k}(\var - (\var^*_{k} - \var_{k}))$.
Then we have
$g_{k} = \grad \tilde{h}_{k}(\var^*_{k})$.
Using the convexity of $\tilde{h}$ and \Cref{lem:lazy}, we have
\begin{align*} 
	\sum_{k=1}^{t} \sigma_{k} \left(\tilde{h}_{k}(\var^*_{k}) - \tilde{h}_{k}(\var^*)\right)
	\le \sum_{k=1}^{t} \sigma_{k} g_{k}^{\top} (\var^*_{k} - \var^*)  
	\le \sum_{k=1}^{t} \frac{2 \sigma_{k}^2 \norm{g_{k}}_2^2}{\rho} + R^2.
\end{align*}
We have $|h_{k}(\var) - \tilde{h}_{k}(\var)| \le L \norm{\var_{k} - \var^*_{k}}_2$
for any $\var \in \cohull$ because $h_{k}$ is $L$-Lipschitz.
Hence
\begin{align*}
	\sum_{k=1}^{t} \sigma_{k} \left(h_{k}(\var_{k}) - h_{k}(\var^*)\right)
	&= \sum_{k=1}^{t} \sigma_{k} \left((h_{k}(\var_{k}) - h_{k}(\var^*_{k})) + h_{k}(\var^*_{k}) - h_{k}(\var^*)\right)\\
  &\le \sum_{k=1}^{t} \sigma_{k} \left(L \norm{\var_{k} - \var^*_{k}}_2
		+ (\tilde{h}_{k}(\var^*_{k}) + \norm{\var_{k} - \var^*_{k}}_2) 
	- (\tilde{h}_{k}(\var^*) - \norm{\var_{k} - \var^*_{k}}_2)\right)\\
  &\le \sum_{k=1}^{t} \left(3L \sigma_{k} \norm{\var_{k} - \var^*_{k}}_2
		+ \frac{2 \sigma_{k}^2 L^2}{\rho}\right) + R^2.
\end{align*}
We have 
$\frac{\rho}{2} \norm{\var_{k} - \var^*_{k}}_2^2 
\le e_{k}
= F_{k}(\var_{k}) - F_{k}(\var^*_{k})$, because $F_{k}$ is
$\rho$-strongly convex. 

Note that $\grad F_{k}(\var_{k}) = g_{k} + \grad \Phi(\var_{k})$,
thus $\var_{k+1}$ is obtained by doing an $\frac{(\lambda - 1)R^2}{\rho}$-FW
step on $F_{k}$.
On the other hand $F_{k}$ has curvature at most $\frac{2 R^2}{\rho}$ 
because $F_{k}$ is 1-smooth, and
	$\norm{\var - \var'}_2^2 
	\le \frac{2}{\rho}(\Phi(\var) - \Phi(\var')) 
	\le \frac{2 R^2}{\rho}$
for any $\var, \var' \in \cohull$ due to the $\rho$-strong convexity of $\Phi$.
Using \Cref{lem:approx-recurrence}, we have
$F_{k}(\var_{k+1}) - F_{k}(\var^*)
	\le (1 - \eta_{k}) e_{k} + \frac{\lambda R^2}{\rho} \eta_{k}^2$.
Thus we have
\begin{align*}
  e_{k+1} 
    &= F_{k}(\var_{k+1}) - F_{k}(\var^*_{k+1}) 
      + \sigma_{k} g_{k} (\var_{k+1} - \var^*_{k+1}) \\
    &\le F_{k}(\var_{k+1}) - F_{k}(\var^*_{k})
      + \sigma_{k} L \norm{\var_{k+1} - \var^*_{k+1}}_2 \\
    &\le (1 - \eta_{k}) e_{k}
      + \frac{\lambda R^2}{\rho} \eta_{k}^2
			+ \sqrt{\frac{2}{\rho}} \sigma_{k} L \sqrt{e_{k+1}}.
\end{align*}
Using \Cref{lem:arsfw-recurrence}, we have $e_{k} \le K \eta_{k}$.
Thus $\norm{\var_{k} - \var_{k}^*}_2 \le \sqrt{\frac{2K}{\rho} \eta_{k}}$.
Hence we have
\begin{align*}
\sum_{k=1}^{t} \sigma_{k} (h_{k}(\var_{k}) - h_{k}(\var^*) )
	&\le \sum_{k=1}^{t} \left(3L \sigma_{k} \norm{\var_{k} - \var^*_{k}}_2
		+ \frac{2 \sigma_{k}^2 L^2}{\rho}\right) + R^2 \\
	&\le \sum_{k=1}^{t} \left(3L \sqrt{\frac{2K}{\rho}} \sigma_k \sqrt{\eta_{k}}
		+ \frac{2 c L^2}{\rho} \sigma_{k} \sqrt{\eta_k} \right)
		+ R^2.  \numberthis  \label{eq:a}
\end{align*}
We used the fact that 
$\sigma_{k}^2 
= \sigma_{k} c\eta_{k}^{3/2}
\le c \sigma_{k} \sqrt{\eta_{k}}$ in the last inequality.
Now observe that we have
\begin{align*} 
	&\quad \E \left(\sum_{k=1}^{t} \sigma_{k} (h_{k}(\var_{k}) - h_{k}(\var^*))\right) \\
	&= \sum_{k=1}^{t} \E\left(\sigma_{k} (h_{k}(\var_{k}) - h_{k}(\var^*))\right) 
	= \sum_{k=1}^{t} \E\left(\sigma_{k} (f(\var_{k}) - f(\var^*))\right) \\
	&= \E\left(\sum_{k=1}^{t} \sigma_{k} (f(\var_{k}) - f(\var^*))\right)
	\ge \E\left((\sum_{k=1}^{t} \sigma_{k}) (f(\bar{\var}_t) - f(\var^*))\right),
	\numberthis \label{eq:b}
\end{align*}
where the first equality holds due to linearity of expectation, 
the second equality holds because we take expectation with respect to $i_{k}$ (but
not $w_{k}$),
the third equality holds due to linearity of expectation, and
the last inequality holds due to the convexity of $f$.
From \Cref{eq:a} and \Cref{eq:b}, we obtain \Cref{eq:rate1}.

When $p = \frac{1}{2}$, observe that
$\sum_{k=1}^{t} \sigma_{k} 
= \sum_{k=1}^{t} c k^{-3/4} 
\ge c t^{1/4}$,
$\sum_{k=1}^{t} \sigma_{k} \sqrt{\eta_{k}}
= \sum_{k=1}^{t} k^{-1} 
\le \ln t +1$, then using \Cref{eq:rate1} we obtain \Cref{eq:rate2}.

When $\sigma_{k} = \sigma = c t^{-3/4}$, 
using \Cref{eq:a} and \Cref{eq:b} and observe that 
$\sum_{k=1}^{t} \sqrt{\eta_{k}} \le \frac{4}{3} t^{3/4}$, we obtain
\Cref{eq:rate3}.
\end{proof}

\section{Conclusion}
We have given a unified analysis of two approximate greedy algorithms, and
presented new results on convergence and their connections.
In addition, we studied their stochastic versions and demonstrated these
algorithms can be robust against the optimization error in each iteration.

There are a few questions for further exploration.
From recent results in FW and the equivalence result in \Cref{thm:equiv}, 
it is natural to ask whether Jones' algorithm converges at faster rates under
suitable additional assumptions, and whether more efficient stochastic Jones'
algorithm can be obtained.
For stochastic FW, the nonsmooth case seems to be harder than the smooth case.
Results on complexity lower bounds will lead to better understanding on the
greedy algorithms and these problems.

\newpage
{\footnotesize
\bibliographystyle{plainnat}
\bibliography{ref}
}

\newpage
\section*{Supplementary Material}
\lemapproxrecurrence*
\begin{proof}
First consider the $\epsilon_{k}$-greedy algorithm. We have 
\begin{align*}
  f\left((1-\eta_{k}) \var_{k} + \eta_{k} d_{k} \right) 
    &\le \min_{d' \in \atoms} f\left((1 - \eta_{k}) \var_{k} + \eta_{k} d'\right) + \epsilon_{k} \\
    &\le \min_{d' \in \atoms}
      \left(f(\var_{k}) + \eta_{k} \grad f(\var_{k})^\top (d' - \var_{k}) + \frac{M}{2} \eta_{k}^2 \right)
      + \epsilon_{k} \\
    &\le f(\var_{k}) + \eta_{k} \min_{d' \in \atoms} \grad f(\var_{k})^\top (d' - \var_{k}) 
      + \epsilon'_{k} \\
    &\le f(\var_{k}) + \eta_{k} (f(\var^*) - f(\var_{k})) + \epsilon'_{k}.
\end{align*}
Let $e_{k} = f(\var_{k}) - f(\var^*)$, and subtract both sides of the above inequality
by $f(\var^*)$, we obtain
\begin{align*}
  e_{k+1} \le (1 - \eta_{k}) e_{k} + \epsilon'_{k}.
\end{align*}

For the $\epsilon_{k}$-greedy FW algorithm, we have
\begin{align*}
  f\left((1 - \eta_{k}) \var_{k} + \eta_{k} d_{k}\right) 
    &\le f(\var_{k}) + \eta_{k} \grad f(\var_{k})^\top (d_{k} - \var_{k}) + \frac{M}{2} \eta_{k}^2 \\
    &\le f(\var_{k}) + \eta_{k} (\min_{d' \in \atoms} \grad f(\var_{k})^\top (d_{k} - \var_{k}) + \epsilon_{k})
      + \frac{M}{2} \eta_{k}^2 \\
    &\le f(\var_{k}) + \eta_{k} \left(f(\var^*) - f(\var_{k})\right) + \eta_{k} \epsilon_{k}
      + \frac{M}{2} \eta_{k}^2.
\end{align*}
Subtracting both sides of the inequality by $f(\var^*)$, we obtain
\begin{align*}
  e_{k+1} \le (1 - \eta_{k}) e_{k} + \epsilon'_{k}.
\end{align*}
\end{proof}

\thmrate*
\begin{proof}
Let $C = 2M+4c$. 
From \Cref{lem:approx-recurrence}, for both $c$-greedy and $c$-FW algorithms, we have
\begin{align*}
	e_{k+1} \le (1 - \eta_{k}) e_{k} + \frac{C}{4} \eta_{k}^2.
\end{align*}
We prove the bound by induction.
Taking $k = 0$, we obtain $f(\var_{1}) - f(\var^*) \le \frac{C}{4} < \frac{C}{k+2}$.

For the inductive, assume the bound holds for $k$, that is, 
$e_{k} \le \frac{C}{k+2}$, then we have
\begin{align*}
	e_{k+1}
    \le \left(1 - \frac{2}{k+2}\right) \frac{C}{k+2} + \frac{C}{(k+2)^2} 
    = \frac{(k+1) C}{(k+2)^2}
    < \frac{C}{k+3},
\end{align*}
where the last inequality holds because 
$(k+1) (k+3) = k^2 + 4k + 3 < (k+2)^2$.
\end{proof}

\thmequiv*
\begin{proof}
\noindent(a)
Suppose the current iterate is $\var_k \in \cohull$. 
If a $\epsilon_{k}$-greedy algorithm yields $(\eta_k, d_k) \in [0, 1] \times \atoms$, then 
\begin{align*}
	&\quad f(\var_{k}) + \eta_{k} \grad f(\var_{k})^\top (d_{k} - \var_{k})  \\
	&\le f\left((1-\eta_{k}) \var_{k} + \eta d_{k} \right)
    && \text{(convexity)} \\
	&\le \min_{d' \in \atoms} f\left((1 - \eta_{k}) \var_{k} + \eta_{k} d'\right) 
		+ \eta_{k} \epsilon_{k}
		&& (\text{$\epsilon_{k}$-greedy}) \\
	&\le \min_{d' \in \atoms} \left(f(\var_{k}) 
		+ \eta_{k} \grad f(\var_{k})^\top (d'-\var_{k}) 
		+ \frac{M}{2} \eta_{k}^2\right) 
		+ \eta_{k} \epsilon_{k}
     && (\text{curvature assumption})\\
	&= f(\var_{k}) 
		+ \eta_{k} \min_{d' \in \atoms} \grad f(\var_{k}) (d' - \var_{k}) 
		+ \eta_{k} \left(\frac{M}{2} \eta_{k} + \epsilon_{k}\right) 
\end{align*}
That is, we have
\begin{align*}
	f(\var_{k}) + \eta_{k} \grad f(\var_{k})^\top (d_{k} - \var_{k}) 
    &\le f(\var_{k}) 
			+ \eta_{k} \min_{d' \in \atoms} \grad f(\var_{k}) (d' - \var_{k}) 
			+ \eta_{k} \left(\frac{M}{2} \eta_{k} + \epsilon_{k}\right).
\end{align*}
Simplifying the above inequality, we have (assuming $\eta_{k} \neq 0$ if $\var_{k}$ is
not an optimal solution)
\begin{align*}
	\grad f(\var_{k})^\top d_{k}
		\le \min_{d' \in \atoms} \grad f(\var_{k}) d' 
			+ \frac{M}{2} \eta_{k} + \epsilon_{k}.
\end{align*}
The case for $c$-greedy algorithms follow easily.

\medskip\noindent(b)
Suppose the current iterate is $\var_{k} \in \cohull$.
If an algorithm is $\epsilon_{k}$-FW, then it gives a $(\eta_{k}, d_{k}) \in [0, 1] \times \atoms$
such that for any $d' \in \atoms$
\begin{align*}
	&\quad f\left((1-\eta_{k}) \var_{k} + \eta_{k} d_{k}\right)  \\
  &\le f(\var_{k}) 
		+ \eta_{k} \grad f(\var_{k})^\top (d_{k} - \var_{k}) 
		+ \frac{M}{2} \eta_{k}^2 
    && \text{(curvature assumption)}\\
  &\le f(\var_{k}) 
		+ \eta_{k} \grad f(\var_{k})^\top (d'-\var_{k}) 
		+ \eta_{k} \epsilon_{k}
		+ \frac{M}{2} \eta_{k}^2 
    && \text{($c$-FW)} \\
  &\le f\left((1-\eta_{k}) \var_{k} + \eta_{k} d'\right) + 
		+ \eta_{k} (\epsilon_{k} + \frac{M}{2} \eta_{k})
    && \text{(convexity)}
\end{align*}
The case for $c$-greedy FW algorithms follow easily.
\end{proof}

\propasjeg*
\begin{proof}
Consider least squares regression 
$\min_{\var \in \cohull} \sum_{i=1}^{3} (x_i^{\top} \var - y_i)^2$, where 
$x_1 = (0, 1)$, 
$x_2 = (-\frac{\sqrt{3}}{2}, -\frac{1}{2})$,
$x_3 = (\frac{\sqrt{3}}{2}, -\frac{1}{2})$, 
$y_1 = 1, y_2 = 1, y_3 = 1$, and
$\atoms = \{x_1, x_2, x_3\}$.

It can be shown that $(\eta_k, d_k) = (1, x_i)$, that is,
$\var_{k+1} = x_i$.
\end{proof}

\thmasj*
\begin{proof}
Let	$u_{k} = \argmin_{d \in \atoms} f((1 - \eta_{k}) \var_{k} + \eta_{k} d_{k})$.
We have
\begin{align*}
	&\quad \eta_{k} \grad \tilde{f}_{k}(\var_{k})^{\top} (d_{k} - \var_{k}) \\
		&= -\tilde{f}_{k}(\var_{k})
			+ \tilde{f}_{k}(\var_{k})
			+ \eta_{k} \grad \tilde{f}_{k}(\var_{k})^{\top} (d_{k} - \var_{k}) \\
		&\le -\tilde{f}_{k}(\var_{k})
			+ \tilde{f}_{k}\left(\var_{k} + \eta_{k}(d_{k} - \var_{k})\right)
			&& (\text{convexity}) \\
		&\le -\tilde{f}_{k}(\var_{k})
			+ \tilde{f}_{k}\left(\var_{k} + \eta_{k}(u_{k} - \var_{k})\right) 
			+ c \eta_{k}^2
			&& (\text{approximate optimality of $d_{k}$}) \\
		&\le -\tilde{f}_{k}(\var_{k})
			+ \tilde{f}_{k}(\var_{k}) 
			+ \eta_{k} \grad \tilde{f}_{k}(\var_{k})^{\top} (u_{k} - \var_{k})
			+ \frac{M+2c}{2} \eta_{k}^2 
			&& (\text{curvature assumption}) \\
		&= \eta_{k} \grad \tilde{f}_{k}(\var_{k})^{\top} (u_{k} - \var_{k})
			+ \frac{M+2c}{2} \eta_{k}^2. \numberthis \label{eq:asj1}
\end{align*}
We have
\begin{align*}
	&\quad f(\var_{k+1}) \\
		&\le f(\var_{k}) 
			+ \eta_{k} \grad f(\var_{k})^{\top} (d_{k} - \var_{k}) 
			+ \frac{M}{2} \eta_{k}^2 \\
		&= f(\var_{k}) 
			+ \eta_{k} \grad \tilde{f}_{k}(\var_{k})^{\top} (d_{k} - \var_{k})
			+ \eta_{k} \left(\grad f(\var_{k}) - \grad \tilde{f}_{k}(\var_{k})\right)^{\top} (d_{k} - \var_{k}) 
			+ \frac{M}{2} \eta_{k}^2 \\
		&\le f(\var_{k})
			+ \left(\eta_{k} \grad \tilde{f}_{k}(\var_{k})^{\top} (u_{k} - \var_{k})
			+ \frac{M+2c}{2} \eta_{k}^2\right)
			+ \eta_{k} \left(\grad f(\var_{k}) - \grad \tilde{f}_{k}(\var_{k})\right)^{\top} (d_{k} - \var_{k}) 
			+ \frac{M}{2} \eta_{k}^2 \\
		&= f(\var_{k})
			+ \eta_{k} \grad f(\var_{k})^{\top} (u_{k} - \var_{k})
			+ \eta_{k} \left(\grad f(\var_{k}) - \grad \tilde{f}_{k}(\var_{k})\right)^{\top} (d_{k} - u_{k}) 
			+ (M + c) \eta_{k}^2 \\
		&\le f(\var_{k})
			+ \eta_{k} (f(\var^*) - f(\var_{k}))
			+ \eta_{k} \left(\grad f(\var_{k}) - \grad \tilde{f}_{k}(\var_{k})\right)^{\top} (d_{k} - u_{k}) 
			+ (M + c) \eta_{k}^2 \\
		&\le f(\var_{k})
			+ \eta_{k} (f(\var^*) - f(\var_{k}))
			+ \eta_{k} \norm{\grad f(\var_{k}) - \grad \tilde{f}_{k}(\var_{k})}_2 D 
			+ (M + c) \eta_{k}^2,
\end{align*}
where the first inequality is due to the curvature assumption,
the second inequality due to \Cref{eq:asj1}, 
the third due to the duality bound,
and the last due to Cauchy-Schwarz and $\norm{d_{k} - u_{k}}_2 \le D$.
Now using convexity and telescoping the above inequality over $k$, we have
\begin{align*}
	&\quad (\sum_{k=1}^{t} \eta_{k}) \left(f(\bar{\var}_{k}) - f(\var^*)\right) \\
	&\le \sum_{k=1}^{t} \eta_{k} (f(\var_{k}) - f(\var^*)) \\
	&\le f(\var_{1}) - f(\var_{t+1}) 
			+ \sum_{k=1}^{t} \eta_{k} D \norm{\grad f(\var_{k}) - \grad \tilde{f}_{k}(\var_{k})}_2
			+ \sum_{k=1}^{t} (M + c) \eta_{k}^2.
\end{align*}
Using $f(\var_{t+1}) \ge f(\var^*)$, 
taking expectation, 
and using 
$\E \norm{\grad f(\var_{k}) - \grad \tilde{f}_{k}(\var_{k})}_2 
\le \frac{L}{\sqrt{b}}$,\footnote{
This is because 
$\left(\E \norm{\grad f(\var_{k}) - \grad \tilde{f}_{k}(\var_{k})} \right)^2
\le \E \norm{\grad f(\var_{k}) - \grad \tilde{f}_{k}(\var_{k})}_2^2
= \frac{1}{b} \E \norm{\grad f_i(\var_{k}) - \grad f(\var_{k})}_2^2
\le \frac{1}{b} \E \norm{\grad f_i(\var_{k})}_2^2
\le \frac{L^2}{b}$, where 
$i$ is randomly drawn from $[n]$.}
we obtain
\begin{align*}
	&\quad \E f(\bar{\var}_{t}) - f(\var^*)
	\le \frac{f(\var_1) - f(\var^*)}{\sum_{k=1}^{t} \eta_{k}}
		+ \frac{\sum_{k=1}^{t} \eta_{k} b_k^{-1/2}}{\sum_{k=1}^{t} \eta_{k}} DL
		+ \frac{\sum_{k=1}^{t} \eta_{k}^2}{\sum_{k=1}^{t} \eta_{k}} (M+c).
\end{align*}

When $b_k = t$ and $\eta_k = t^{-1/2}$ for all $k$, we have 
\begin{align}
	\E f(\bar{\var}_{k}) - f(\var^*)
	\le \frac{f(\var_1) - f(\var^*) + DL + M + c}{\sqrt{t}}.
\end{align}

When $b_k = k$, and $\eta_{k} = k^{-1/2}$, we have 
\begin{align}
	\E f(\bar{\var}_{k}) - f(\var^*)
	\le \frac{f(\var_1) - f(\var^*) + (DL + M + c) (\ln t + 1)}{\sqrt{t}}.
\end{align}
\end{proof}

\propasfwega*
\begin{proof}
Consider 
$\min_{\norm{\var}_2 \le r} \sum_{i=1}^2 (x_i^\top \var - y_i)^2$, where
$x_1 = x_2 = x = (1,1)$, $y_1 = 1$, $y_2 = -1$, and $r = 1/2$.
Here we can take $\atoms = \{\var: \norm{\var}_2 = r\}$.

We first show prove convergence.
Let $X_i$ be the random variable taking value 1 when $(x_1, y_1)$ is sampled
at iteration $i$, and value -1 otherwise.
Define $Y_{k+1} = (1 - \gamma_i) Y_{k} + \gamma_{k} X_{k}$, then
it can be verified that $\var_{k} = Y_{k} x$.
In addition, we can show that $Y_{k}$ converges in probability to 0, which
implies that $Y_{k} x$ converges in probability to a minimizer $\var^* =
(0,0)$ of $f(\var)$, and thus $\E(f(\var_{k})) - f(\var^*)$ converges to 0.

We prove the concentration result of $Y_{k}$ for the more general case where
$X_i$'s are i.i.d. drawn from a distribution on $[a, b]$ with mean $\mu$,
instead of from the uniform distribution on \{-1, 1\}.
First we have $Y_{k} = \sum_{i=0}^{k} w_i X_i$, where 
$w_i = \frac{2}{i+2} \frac{i+1}{i+3} \ldots \frac{k}{k+2} \le \frac{2}{k+2}$.
By Hoeffding's inequality, we have
\begin{align*}
  P(|Y_{k} - \mu| \ge \epsilon) 
    &\le 2 e^{-2 \epsilon^2 / \sum_{i=0}^t (w_i b -  w_i a)^2}. 
\end{align*}
Since each $w_i \le \frac{2}{k+2}$, we have
\begin{align*}
  P(|Y_{k} - \mu| \ge \epsilon) 
    &\le 2 e^{-(k+2) \epsilon^2 / 2(b - a)^2}.
\end{align*}

We have 
$\grad f_i(\var_{k}) = 2 (x_i^\top \var_{k} - y_i) x_i$,
and $\grad f(\var_{k}) = 4 (x^\top \var_{k}) x$.
In addition, 
let $d_{k,i} = \argmin_{d \in \atoms} \langle d, \grad f_i(\var_{k})$, then
\begin{equation*}
  d_{k,i} 
  = -r \sign(x_i^\top \var_{k} - y_i) x_i / \norm{x_i}_2
  = r \sign(y_i) x_i / \norm{x_i}_2
  = \frac{\sign(y_i)}{2\sqrt{2}} x,
\end{equation*}
where the $\sign(x_i^\top \var_{k} - y_i) = \sign(y_i)$ because 
$x_i^\top \var_{k}$ is not large enough to change the sign of $y_i$.
Hence we have
\begin{equation*}
  \E(d_{k}) = \frac{1}{2} d_{k,1} + \frac{1}{2} d_{k,2}  = 0.
\end{equation*}
On the other hand, we have
\begin{equation*}
  \argmin_{\norm{d}_2 \le r} \langle d, \grad f(\var_{k}) \rangle
    = -r \sign(x^\top \var_{k}) x / \norm{x}_2.
\end{equation*}
We thus obtain
\begin{align*}
  \langle \E(d_{k}), \grad f(\var_{k}) \rangle 
    - \min_{d \in \atoms} \langle s, \grad f(\var_{k}) 
  = 0 - \langle - r \sign(x^\top \var_{k}) x / \norm{x}_2, 4 (x^\top \var_{k}) x \rangle 
  = 4r |x^\top \var_{k}| \norm{x}_2.
\end{align*}
Note that $\var_{k}$ has nonzero probability of being $x$, thus there is a
nonzero probability that the above difference equals 
$4 r |x^\top \var_{k}| \norm{x}_2 = 4\sqrt{2}$.
However $\eta_{k} = 2/(k+2)$ converges to 0, thus there is no constant $c$
such that
\begin{equation*}
  \langle \E(d_{k}), \grad f(\var_{k}) \rangle 
    - \min_{d \in \atoms} \langle d, \grad f(\var_{k}) \\
  \le c \eta_{k}
\end{equation*}
for all $k$ and all $\var_{k}$.
\end{proof}

\propasfwegb*
\begin{proof}
Consider least squares regression 
$\min_{||\var||_2 \le r} \sum_{i=1}^{3} (x_i^T \var - y_i)^2$, where 
$r < 1$, $x_1 = (0, 1)$, 
$x_2 = (-\frac{\sqrt{3}}{2}, -\frac{1}{2})$,
$x_3 = (\frac{\sqrt{3}}{2}, -\frac{1}{2})$, and 
$y_1 = 1, y_2 = - 1, y_3 = -1$.

It can be shown that $\var_k$ converges in probability to
$(0, 2r/3)$ as $k \to \infty$.
However, the optimal solution is $(0, r)$.
\end{proof}

\lemlazy*
\begin{proof}
We have 
\begin{align*}
	\frac{\rho}{2} \norm{\var^*_{k} - \var^*_{k+1}}_2^2
	&\le F_{k+1}(\var^*_{k}) - F_{k+1}(\var^*_{k+1}) 
		&& (\text{$F_{k+1}$ is $\rho$-strongly convex}) \\
	&= F_{s}(\var^*_{k}) - F_{s}(\var^*_{k+1}) + \sigma_{k} g_{k}^{\top} (\var^*_{k} - \var^*_{k+1}) 
		&& (\text{definition of $F_{k+1}$})\\
	&\le \sigma_{k} g_{k}^{\top} (\var^*_{k} - \var^*_{k+1}) 
		&& (\text{$F_{s}(\var^*_{k}) \le F_{s}(\var^*_{k+1})$}) \\
	&\le \sigma_{k} \norm{g_{k}}_2 \norm{\var^*_{k} - \var^*_{k+1}}_2.
	&& (\text{Cauchy-Schwarz})
\end{align*}
Hence we have $\norm{\var^*_{k} - \var^*_{k+1}}_2 \le \frac{2 \sigma_{k} \norm{g_{k}}_2}{\rho}$, and
this implies 
\begin{align}
		\sigma_{k} g_{k}^{\top} (\var*_{k} - \var^*_{k+1}) \le \frac{2 \sigma_{k}^2 \norm{g_{k}}_2^2}{\rho}.
		\label{eq:lemlazy1}
\end{align}

We claim that 
\begin{align}
  \sum_{i=1}^{s} \sigma_i g_i^{\top} (\var^*_i - \var)  
    \le \sum_{i=1}^{s} \sigma_i g_i^{\top} (\var^*_i - \var^*_{i+1}) + \Phi(\var) - \Phi(\var^*_1).
		\label{eq:lemlazy2}
\end{align}
This is equivalent to
\begin{align*}
  \sum_{i=1}^{s} \sigma_i g_i^{\top} \var^*_{i+1} + \Phi(\var^*_1)
    \le \sum_{i=1}^{s} \sigma_i g_i^{\top} \var + \Phi(\var).
\end{align*}
This holds when $s = 0$ by the definition of $\var^*_1$.
If this holds for some $s$, then this holds for $s+1$ as follows,
\begin{align*}
  \sum_{i=1}^{s+1} \sigma_i g_i^{\top} \var^*_{i+1} + \Phi(\var^*_1)
    \le \sum_{i=1}^{s} \sigma_i g_i^{\top} \var^*_{s+2} + \sigma_{k+1} g_{k+1}^{\top} \var^*_{s+2} + \Phi(\var^*_{k+1})
    \le \sum_{i=1}^{s+1} \sigma_i g_i^{\top} \var + \Phi(\var).
\end{align*}
where the first inequality uses the inductive assumption, and the second one
holds by the definition of $\var^*_{s+2}$.

Combining \Cref{eq:lemlazy1} and \Cref{eq:lemlazy2}, we have
\begin{align*}
	\sum_{k=1}^{t} \sigma_{k} g_{k}^{\top} (\var^*_{k} - \var) 
	\le \sum_{k=1}^{t} \frac{2 \sigma_{k}^2 \norm{g_{k}}_2^2}{\rho} + \Phi(\var) - \Phi(\var^*_1).
	\le \sum_{k=1}^{t} \frac{2 \sigma_{k}^2 \norm{g_{k}}_2^2}{\rho} + R^2.
\end{align*}
\end{proof}

\lemsfwrecurrence*
\begin{proof}
We first transform the recurrence in \Cref{eq:e} in the form 
$e_{k+1} \le h(e_{k})$ for some function $h$.
For nonnegative numbers $A, B, C$ to satisfy $A \le B + C \sqrt{A}$, we need to have
$(\sqrt{A} - \frac{C}{2})^2 \le B + \frac{C^2}{4}$, or
$\sqrt{A} \le \sqrt{B + \frac{C^2}{4}} + \frac{C}{2}$.
This implies $A \le B + \frac{C^2}{2} + C \sqrt{B + \frac{C^2}{4}}$.
Applying this transformation to the recurrence in \Cref{eq:e}, we have
\begin{align}
  e_{k+1} &\le  (1 - \eta_{k}) e_{k} + \frac{\lambda R^2}{\rho} \eta_{k}^2
		+ \frac{\sigma_{k}^2 L^2}{\rho}
		+ \sqrt{\frac{2}{\rho}} \sigma_{k} L \sqrt{
				(1 - \eta_{k}) e_{k} 
				+ \frac{\lambda R^2}{\rho} \eta_{k}^2 
				+ \frac{\sigma_{k}^2 L^2}{2\rho}} \nonumber \\
		&\le A' + \sqrt{\frac{2}{\rho}} \sigma_{k} L \sqrt{A'}
		\label{eq:e_{k}+1}.
\end{align}
where 
$A' = (1 - \eta_{k}) e_{k} 
+ \frac{\lambda R^2}{\rho} \eta_{k}^2 
+ \frac{c^2 L^2}{\rho} \eta_{k}^2$.
The second inequality holds by observing that $\frac{\sigma_{k}^2 L^2}{\rho}$
and $\frac{\sigma_{k}^2 L^2}{2\rho}$ in the first inequality are smaller than
$\frac{c^2 L^2}{\rho}\eta_{k}^2$, because 
$\sigma_{k}^2 = c^2 \eta_{k}^3 \le c^2 \eta_{k}^2$.

Now we find a value of $K$ by determining a sufficient condition on $K$ such
that $e_{k+1} \le K \eta_{k+1}$ holds when $e_{k} \le K \eta_{k}$.
Assume $e_{k} \le K \eta_{k}$ for some $s$, then
\begin{alignat}{3}
	e_{k+1} &\le K \eta_{k+1} 
	&\quad\Longleftarrow\quad&&
		A' + \sqrt{\frac{2}{\rho}} \sigma_{k} L \sqrt{A'} &\le K \eta_{k+1} 
		\label{eq:back1} \\
	&&\quad\Longleftarrow\quad&&
		A' + \sqrt{\frac{2}{\rho}} c L \sqrt{K} \eta_{k}^{2} &\le K \eta_{k+1}
		\label{eq:back2} \\
	&&\quad\Longleftrightarrow\quad&&
		(1 - \eta_{k}) K \eta_{k} + M \eta_{k}^2 &\le K \eta_{k+1}
		\label{eq:back3} \\
	&&\quad\Longleftrightarrow\quad&&
	\frac{K}{K - M} &\le \frac{\eta_{k}^2}{\eta_{k} - \eta_{k+1}} 
		\label{eq:back4} \\
	&&\quad\Longleftarrow\quad&&
	\frac{K}{K - M} &= \frac{1}{p},
		\label{eq:back5}
\end{alignat}
where $M = \frac{\lambda R^2 + c^2 L^2}{\rho} + \sqrt{\frac{2}{\rho}} cL \sqrt{K}$.
\begin{itemize}
\item
When \Cref{eq:back1} holds, then from \Cref{eq:e_{k}+1}, we have
$e_{k+1} \le K \eta_{k+1}$.
\item
When \Cref{eq:back2} holds, we have $K \eta_{k} > K \eta_{k+1} > A'$.
Thus 
$A' + \sqrt{\frac{2}{\rho}} c L \sqrt{K} \eta_{k}^{2} 
= A' + \sqrt{\frac{2}{\rho}} c \eta_{k}^{3/2} L \sqrt{K \eta_{k}}
\ge A' + \sqrt{\frac{2}{\rho}} \sigma_{k} L \sqrt{A'}$.
\item
Both \Cref{eq:back3} and \Cref{eq:back4} are just rewriting of the previous
inequality.
\item
To show that \Cref{eq:back5} implies \Cref{eq:back4}, it suffices to show
that $\frac{\eta_{k}^2}{\eta_{k} - \eta_{k+1}} \ge \frac{1}{p}$ for any
$s \ge 1$.
Using calculus, we have 
$(1 + \frac{1}{k})^p 
\le 1 + \frac{p}{k}
\le 1 + \frac{p}{k^p}$ for $k \ge 1$.
Hence we have 
$\eta_{k} - \eta_{k+1} 
= \frac{1}{k^p} - \frac{1}{(k + 1)^p}
= \frac{1}{(k + 1)^p} \left(\frac{(k+1)^p}{k^p} - 1\right)
= \frac{1}{(k+1)^p} \left((1 + \frac{1}{k})^p - 1\right) 
\le \frac{1}{(k+1)^p} \frac{p}{k^p}$.
It follows that
$\frac{\eta_{k}^2}{\eta_{k} - \eta_{k+1}} \ge \frac{1}{p} \frac{(k+1)^p}{k^{p}} \ge \frac{1}{p}$.
\end{itemize}
Now we solve \Cref{eq:back5}.
This is equivalent to 
$K = \frac{M}{1-p} 
= \frac{\lambda R^2 + c^2 L^2}{\rho(1-p)} 
	+ \sqrt{\frac{2}{\rho}}\frac{cL}{1-p} \sqrt{K}$, 
which is equivalent to 
$(\sqrt{K} - \sqrt{\frac{1}{2\rho}} \frac{cL}{1-p})^2 
	= \frac{\lambda R^2 + c^2 L^2}{\rho(1-p)} 
		+ \frac{c^2 L^2}{2 \rho (1-p)^2}$,
or
\begin{align*}
K = 
	\left(\sqrt{\frac{1}{2\rho}} \frac{cL}{1-p}
		+ \sqrt{\frac{\lambda R^2 + c^2 L^2}{\rho(1-p)} 
		+ \frac{c^2 L^2}{2 \rho (1-p)^2}}\right)^2.
\end{align*}

To complete the proof it suffices to show that $e_{1} \le K \eta_1$.
This holds because 
\begin{align*} 
	e_1 = F_1(\var_1) - F_1(\var^*_1) = \Phi(\var_1) - \Phi(\var^*_1) = 0 < K \eta_1.
\end{align*}
\end{proof}

\end{document}